\documentclass[twoside,11pt,a4paper]{article}
\usepackage{amsmath,float,amssymb,mathtools,amsthm,mathrsfs,multirow,booktabs,setspace,morefloats,epsfig,caption,graphicx,inputenc,subcaption}
\usepackage[longnamesfirst]{natbib}
\bibpunct[, ]{(}{)}{;}{a}{}{,}
\usepackage{bigints}
\usepackage{xcolor}
\usepackage{placeins}
\newtheorem{lemma}{Lemma}
\newtheorem{thm}{Theorem}
\newtheorem{definition}{Definition}
\usepackage{fancyhdr}
\newcommand\shorttitle{Some multivariate goodness of fit tests based on data depth}

\newcommand\authors{Rahul Singh, Subhajit Dutta and Neeraj Misra}

\title{Some multivariate goodness of fit tests based on data depth}
\author{Rahul Singh$^1$, Subhajit Dutta$^2$ and Neeraj Misra$^3$}
\date{
	\small{Department of Mathematics and Statistics, Indian Institute of Technology Kanpur, India}\\%
}

\begin{document}
	\maketitle
	{\let\thefootnote\relax\footnote{{ Email: $^1$sirahul@iitk.ac.in, $^2$duttas@iitk.ac.in, $^3$neeraj@iitk.ac.in}}}
	
	\begin{abstract}
		Using the fact that some depth functions characterize certain family of distribution functions, and under some mild conditions, distribution of the depth is continuous, we have constructed several new multivariate goodness of fit tests based on existing univariate GoF tests. Since exact computation of depth is difficult, depth is computed with respect to a large random sample drawn from the null distribution. It has been shown that test statistic based on estimated depth is close to that based on true depth for a large random sample from the null distribution. 
		Some two sample tests for scale difference, based on data depth are also discussed. These tests are distribution-free under the null hypothesis.
		Finite sample properties of the tests are studied through several numerical examples. A real data example is discussed to illustrate usefulness of the proposed tests.
	\end{abstract}
	
	\section{Introduction}
	Testing goodness of fit (GoF) of a given dataset w.r.t. a given probabilistic model is an essential aspect of any data analysis. Most GoF tests have been developed for univariate distribution functions (except for multivariate normality). Theoretically, Pearson's chi-square test can be used for GoF test of any multivariate distribution function. This requires division of data into disjoint classes. Such a test is sensitive to categorisation and the optimal way to decide the class boundaries is not clear. {\cite{moore1981}} suggested taking class boundaries as concentric hyper-ellipses centred at the sample mean and shape determined by inverse of the covariance matrix. However, theoretical properties of such tests are not known.
	
	Let $X_1,X_2,\ldots,X_n$ be a random sample from a population having absolutely continuous distribution $F$ on $\mathbb{R}^d$, the $d$-dimensional Euclidean space. Let $F_0$ be a specified absolutely continuous distribution $F$ on $\mathbb{R}^d$. Consider testing
	\begin{equation}\label{problem}
		H_0:F=F_0 \text{ against } H_A:F\neq F_0.
	\end{equation}
	\noindent
	Although a large number of goodness of fit tests have been proposed for specific $F_0$ (e.g. multivariate normal), only a few tests have been proposed for general $F_0$. Two popular tests for testing (\ref{problem}) are multivariate Kolmogorov-Smirnov test ({\cite{ks1997}}) and a test based on the empirical characteristic function ({\cite{fan1997}}; {\cite{cfgof2009}}).
	
	{\cite{rosenblatt1952}} proposed a multivariate generalisation of the probability integral transform. {\cite{ks1997}} utilized this to construct a multivariate Kolmogorov-Smirnov (MKS) test for testing (\ref{problem}), and provided an algorithm to implement the test procedure for dimension two. For the multivariate case, owing to the fact that the empirical distribution function has jumps not only at sample points but also at many other points, evaluation of the MKS test statistic in higher dimension can be cumbersome. The MKS test algorithm is not available for dimensions greater than two.
	
	Characteristic function characterizes the distribution function and empirical characteristic function is a consistent estimator of the true characteristic function ({\cite{ecf1977}}). Many authors including {\cite{fan1997}} and {\cite{cfgof2009}}, exploited this to construct a GoF test based on empirical characteristic function. An important advantage of multivariate GoF test based on the empirical characteristic function is that these can be used for discrete distribution functions as well as for composite null hypothesis, but computation of the test statistic and critical value is challenging. Most of these tests are implemented using bootstrapping and the test algorithm is dependent on the null distribution. This makes the test computationally quite expensive.
	
	In the statistics literature, data depth of a data point is a measure of centrality of the point with respect to the data cloud (or, a distribution function), which provides an ordering of observations of a multivariate dataset (\cite{liu1999}). There are various notions of data depth (\cite{zuo2000}), and we are interested in those which give characterization for some family of distribution functions.
	
	In this article, we aim to study GoF tests, based on data depth functions when $F$ is completely specified. {\cite{zhnag2012}} discussed some GoF tests for bivariate uniform and bivariate normal distributions, based on Mahalanobis depth and projection depth. Mahalanobis depth depends on only first two moments of distribution function which may not characterise the distribution function, while projection depth is based on outlyingness, which again depends on some statistics which may not characterize the distribution. So, these tests have very limited applicability. {\citet{li2018}} studied univariate GoF tests based on Tukey's half-space depth (and simplicial depth) and found that these tests perform better than usual GoF tests in case of scale differences.
	
	{\cite{kosh2003}} proved that the Tukey's half-space depth uniquely determines absolutely continuous distribution function with compact support, \cite{hassi2008} found that the Tukey depth characterises absolutely continuous distribution functions with connected support and density function being continuous on the interior of the support, while {\cite{kong2010}} proved that distribution with smooth depth contours is completely determined by its Tukey's half-space depth (e.g., elliptically symmetric distributions). {\cite{mos2012}} showed that zonoid depth characterises general probability distribution with finite first moment.  Since in dimension greater than one, exact value of most of depth functions are difficult to compute, the empirical depth with respect to a large random sample from null distribution, are reasonable approximations. As empirical Tukey's half-space and empirical zonoid depths converge almost surely and uniformly to their population versions respectively ({\cite{zuo2000}, \cite{dyck2016}}), these depth functions are suitable candidates for our GoF tests. 
	
	The above discussed problem of goodness of fit is also known as a one sample problem. Suppose that two samples are given from two unknown distributions. The problem of testing, whether these two distributions are the same or not, is known as the two sample problem. Here, we are interested in testing the scale, or shape difference between populations. Two popular scale difference tests were given by \cite{liu2006} and \cite{li2016}. The Kolmogorov-Smirnov (KS) test, the Cram\'er-von Mises (CvM) test and Anderson-Darling (AD) test have been extended to two sample problems. We study these tests based on data depth and compare their performance with tests proposed by \cite{liu2006} and \cite{li2016}.
	
	The paper is organized as follows. In Section 2, we briefly discuss Tukey depth and zonoid depth, and introduce some new multivariate GoF tests based on data depth and derive their properties. The finite sample performance of the proposed tests is studied through several numerical studies. A summary of the study is given in Section 3,  and conclusion of the study is given in Section 4. All proofs are given in the Appendix.
	\section{One sample GoF tests}
	Data depth of a data point is a measure of centrality of the point with respect to the data cloud (or, a distribution function). There are several notions of data depths studied in literature ({\cite{liu1999}; \cite{zuo2000}}). Here we are interested in Tukey depth ({\cite{t1955}}; {\cite{t1975}}) and zonoid depth ({\cite{zd1997}}), as these two depths satisfy desirable conditions required for characterisation of some families of distribution functions ({\cite{donoho1992}; \cite{dyck2016}}).
	
	\begin{definition}
		Tukey's half-space depth of $x\in \mathbb{R}^d$ with respect to $F$ is defined as
		$$HD_F(x):=\underset{\mathrm{ }}{\inf}\{P_F(H): H \text{ is a closed half space in } \mathbb{R}^d \text{ such that } H \ni x \},$$
		where $P_F$ denotes probability measure corresponding to the distribution function $F$.
	\end{definition}
	Tukey's half-space depth is also known as Tukey's depth or, half space depth.
	\begin{definition} Let $\alpha\in (0,1]$ and suppose $F$ has finite first moment. Take
		$$D_\alpha (F):= \left\{\int_{\mathbb{R}^d}x\ g(x)\ dF(x): g: \mathbb{R}^d\rightarrow \left[0,\frac{1}{\alpha}\right]\ and\ \int_{\mathbb{R}^d} g(x)\ dF(x)=1\right\}. $$
		Zonoid depth of $y\in \mathbb{R}^d$ with respect to $F$ is defined as follows
		$$ZD_F(y):=\begin{cases}
		\sup \{\alpha : y\in D_\alpha (F)\}, &\quad\text{if the set is non empty,}\\
		\text{0,} &\quad\text{otherwise.} \\
		\end{cases}$$
	\end{definition}
	\noindent
	The sample version of these depths are obtained by replacing $F$ by empirical distribution function $F_n$, with sample size $n$.
	
	For the testing problem (\ref{problem}), we draw a large random sample $W_1,W_2,\ldots,W_N$ from $F_0$, independent of $\{X_1,X_2,\ldots,X_n\}$. Let $D(\cdot, F)$ be a depth function with respect to distribution function $F$. Suppose $D(\cdot, F)$ satisfies following assumptions:
	\begin{itemize}
		\item[(A1)] Distribution function of $D(\cdot, F)$ is continuous.
		\item[(A2)] The sample version converges almost surely and uniformly to the population version.
	\end{itemize}
	Many commonly used depths satisfy assumptions (A1) and (A2). For absolutely continuous distribution functions, Tukey's depth
	satisfies assumptions (A1) and (A2) (\cite{donoho1992}; \cite{mass2002}). Zonoid depth also satisfies assumption (A1) and (A2), for absolutely continuous distribution functions (\cite{mos2012}; {\cite{zonoid2016}; \cite{dyck2016}}).
	
	Consider the following notation:\\
	\noindent
	$D\ \ \ :=\ D(\cdot, F_0)$, depth function with respect to distribution function $F_0$.\\
	$D_N\ :=\ D(\cdot, F_N)$, empirical depth function with respect to sample $\{W_1,W_2,\ldots,W_N\}$.\\
	$F^D\ :=\ $ distribution function of $D(W_1)$.\\
	$F_N^D\ :=\ $ ecdf of $\{D(W_1),D(W_2),\ldots,D(W_N)\}$.\\
	$G_N(x)\ :=\ \dfrac{1}{N}\sum_{i=1}^{N} \mathbb{I}(D_N(W_i)\leq D_N(x)).$\\
	$\mathbf{F}^D_{N,n}=:\pmb(F^D(D(X_1)),F^D(D(X_2)),\ldots,F^D(D(X_n))\pmb)^T$.\\
	$\mathbf{G}^D_{N,n}:=\pmb(G_N(X_1),G_N(X_2),\ldots,G_N(X_n)\pmb)^T$.
	\begin{thm}[\cite{liu1993}]\label{thm1}
		Under the null hypothesis $H_0$ and assumption (A1),
		\begin{align*}
			\mathbf{F}^D_{N,n}\stackrel{d}{=} (U_1,U_2,\ldots,U_n),\text{ where }U_i\stackrel{iid}{\sim}U[0,1].
		\end{align*}
	\end{thm}
	Using {Theorem \ref{thm1}}, under the null hypothesis $H_0$, distribution of any measurable function of $F^D_{N,n}$ doesn't depend on $F$. When $D$ determines the distribution function uniquely, the testing problem reduces to
	\begin{align}
		H_0:\ F^D(X_i)\stackrel{iid}{\sim}\ U[0,1] \text{ against } H_A:\ F^D(Y_i)\stackrel{iid}{\nsim}\ U[0,1],~i=1,2,\ldots,n.
	\end{align}
	Several GoF tests for $U[0,1]$ are known to have good properties e.g. Kolmogorov-Smirnov (KS), Anderson-Darling (AD), Cramer-von Mises (CvM), Greenwood tests, etc. (see {\cite{d1986}}). Since exact computation of most of depth functions is difficult, $F^D_{N,n}$ is unobservable. Depth functions satisfying assumptions (A1) and (A2),
	$G^D_{N,n}$ can be used as an approximation of $F^D_{N,n}$ for large $N$. So, $G^D_{N,n}$ is close to a random sample from $U[0,1]$, and hence can be used for the testing (\ref{problem}). We have utilised this idea to construct new GoF tests for absolutely continuous distribution functions based on data depth.
	
	Suppose we have a random sample $X_1,X_2,\ldots,X_n$ from $F$ and wish to test $H_0:F=F_0$ against $H_A:F\neq F_0$. The steps of proposed tests are as follows.
	\subsection*{Steps of Test}
	\begin{itemize}
		\item[1. ]  Fix a large $N$ and take a large sample $W_1,W_2,\ldots,W_N$ from $F_0$. Let $D(\cdot,F_0)$ be Tukey's depth or zonoid depth with respect to the distribution function $F_0$ and $D_N(\cdot)\ =D(\cdot,F_N)$ be the corresponding sample depth with respect to the sample $\{W_1,W_2,\ldots,W_N\}$.
		\item[2. ] Compute $\{D_N(W_1),D_N(W_2),\ldots,D_N(W_N)\}$ and $\{D_N(X_1),D_N(X_2),\ldots,D_N(X_n)\}$.
		\item[3. ] Compute $G_N(X_j)=\ \dfrac{1}{N}\sum_{i=1}^{N} \mathbb{I}(D_N(W_i)\leq D_N(X_j))$ for $j= 1,2,\ldots,n$.
		\item[4. ] Let $G^{(1)}, G^{(2)}, \ldots,G^{(n)} $ be the order statistics corresponding to $G_N(X_j)$, $j= 1,2,\ldots,n$. Then, the proposed depth based test statistics are given by the following expressions,
		\begin{table*}[htb]
			\centering
			\begin{tabular}{c c }
				KS type test statistic & $dKS_n= \underset{u\in [0,1]}{\sup}\bigg|\dfrac{1}{n}\sum_{j=1}^{n}\mathbb{I}\left(G^{(j)}\leq u\right)-u\bigg|$\\~\\
				CvM type test statistic & $dCvM_n= \frac{1}{12n}+\sum_{j=1}^{n}\pmb[\frac{2j-1}{2n}-G^{(j)}\pmb]$\\~\\
				AD type test statistic & $dAD_n= -n-\sum_{j=1}^n\frac{2j-1}{n}\pmb[\log(G^{(j)})-\log(G^{(n-j+1)})\pmb]$\\~\\
				Greenwood type test statistic & $dGD_n=\frac{1}{n} \sum_{j=1}^{n}\pmb[n(G^{(j)}-G^{(j-1)})\pmb]^2$.
			\end{tabular}
		\end{table*}
		\FloatBarrier
		\item[5.] The computed test statistic in the last step, can be compared with the critical value of the corresponding test statistic for $U[0,1]$ to arrive at a decision.
	\end{itemize}
	Our test procedure translates the multivariate GoF problem to testing uniformity on $[0,1]$. An advantage of this procedure is that it remains unaffected by the dimension of the data.
	
	\subsection{Theoretical Results}
	Using Theorem \ref{thm1}, any statistic based on $\mathbf{F}^D_{N,n}$ is distribution-free. For testing uniformity, several tests based on $\mathbf{F}^D_{N,n}$ are known to have good properties, e.g., KS, AD, CvM, Greenwood test statistics, etc. (see {\cite{d1986}}). Due to practical considerations, we propose to use $\mathbf{G}^D_{N,n}$ in place of $\mathbf{F}^D_{N,n}$. The following result gives a relation between these two.
	
	\begin{thm}\label{thm2}
		Under assumptions (A1) and (A2), for finite n,
		$$||\mathbf{F}^D_{N,n}-\mathbf{G}^D_{N,n}||\xrightarrow[N\to \infty]{a.s.}\ 0,$$
		where $||\cdot||$ denotes the Euclidean norm.
	\end{thm}
	{Theorem \ref{thm2}} says that for finite $n$, the distance between $\mathbf{G}^D_{N,n}$ and $\mathbf{F}^D_{N,n}$ converges almost surely to zero as $N\to \infty$.
	\begin{thm}\label{thm3}
		Under assumptions (A1) and (A2), and for finite $n$ , test statistics $dKS_n$, $dCvM_n$, $dAD_n$ and $dGD_n$ converge almost surely to KS, CvM, AD and Greenwood test statistics based on $\mathbf{F}^D_{N,n}$ respectively, as $N\to \infty$.
	\end{thm}
	Using {Theorem \ref{thm3}}, we have for finite $n$ some common test statistics based on $\mathbf{G}^D_{N,n}$ converge almost surely to those based on $\mathbf{F}^D_{N,n}$ as $N\to \infty$.
	
	If the depth $D$ determines distribution function uniquely, then KS, AD and CvM tests based on $\mathbf{F}^D_{N,n}$ are consistent for $F\neq F_0$ (see, e.g., {\cite{dasgupta2008}}). Also if depth $D$ determines distribution function uniquely, then Greenwood tests based on $\mathbf{F}^D_{N,n}$ is asymptotically locally most powerful among symmetric spacings tests based on $\mathbf{F}^D_{N,n}$ ({\cite{sethu1970}}). Spacings based on $\mathbf{F}^D_{N,n}$ correspond to depth spacings defined by {\cite{li2008}}.
	
	Observe that {Theorems \ref{thm2} and \ref{thm3}} hold even if $d>n$, which is unusual in multivariate procedures. Because the depth is computed with respect to simulated sample $W_1,W_2,\ldots,W_N$ and {Theorems \ref{thm2}-\ref{thm3}} require $N\to\infty$, and we can choose $N$. Generally, multivariate procedures are valid only for $d<n$.
	\vspace*{-0.1in}
	\section{Two sample tests}
	In the multivariate setup, \cite{liu1993} introduced the rank sum test based on data depth. \cite{liu1999} proposed depth-depth plot and thereafter, many two sample tests were proposed in the literature based on depth-depth plot (\cite{li2004}, \cite{dovo2014}). These tests are mostly based on permutations. Depth-depth plot based tests are useful to detect change in location or scale. There are some other depth based tests discussed in the literature (\cite{liu1993}, \cite{chenori2011}). \cite{li2011} discussed some statistics of type $KS$ and $CvM$ and permutation based tests thereof.
	
	Let $X_1,X_2,\ldots,X_n$ and $Y_1,Y_2,\ldots,Y_m$ be independent random samples from $d$-dimensional unknown continuous distributions $F$ and $G$, respectively. We aim to test the null hypothesis
	\begin{align*}
		H_0: F(x)=G(x)\ \forall x\in\mathbb{R}^d \text{ against }H_A:F(x)\neq G(x)\text{ for some } x\in\mathbb{R}^d.
	\end{align*}
	\noindent
	If there is a scale difference between distributions $F$ and $G$ (in the sense that $G$ has higher dispersion than $F$), then $X_i$'s are more likely to be concentrated near geometrical centre of the joint sample and $Y_i$'s are more likely to be spread towards outer positions of the cluster of the joint data. This observation guides us to construct tests for scale difference by using some measure of outlyingness of sample points with respect to a data cloud.
	
	Let us denote the joint sample as $Z_1,Z_2,\ldots,Z_N$, where $N=n+m$. Let $D$ be a given depth function, and $D(\cdot,H_N)$ denote the depth with respect to the joint sample. Denote:
	\begin{itemize}
		\item[ ] $X_i^*=D(X_i,H_N)$ for $i=1,2,\ldots,n$;
		\item[ ] $Y_i^*=D(Y_i,H_N)$ for $i=1,2,\ldots,m$ and
		\item[ ] $Z_i^*=D(Z_i,H_N)$ for $i=1,2,\ldots,N$.
	\end{itemize}
	Define $F_n^*$, $G_m^*$ and $H_N^*$ as
	\begin{align*}
		F_n^*(t)=&\frac{1}{n}\sum_{i=1}^{n}\mathbb{I}(X_i^*\leq t),~~
		G_m^*(t)=\frac{1}{m}\sum_{i=1}^{m}\mathbb{I}(Y_i^*\leq t),\\ 
		\text{ and }H_N^*(t)=&\frac{1}{N}\sum_{i=1}^{N}\mathbb{I}(Z_i^*\leq t),\ t\in\mathbb{R}.
	\end{align*}
	We can define two sample $KS$, $CvM$ and $AD$ statistics based on $F_n^*$ and $G_m^*$, as follows.
	\begin{itemize}
		\item[(i) ] Two sample KS test statistic based on data depth:
		\begin{align*}
			\tilde{dKS}=&\sup_{t\in\mathbb{R}}\big|F_n^*(t)-G_m^*(t)\big|=\max_{1\leq i\leq N}\big|F_n^*(Z_i^*)-G_m^*(Z_i^*)\big|.
		\end{align*}
		\item[(ii) ] Two sample CvM test statistic based on data depth:
		\begin{align*}
			\tilde{dCvM}=&\frac{nm}{n+m}\int_{-\infty}^{\infty}\left(F_n^*(t)-G_m^*(t)\right)^2dH_N^*(t)
			=\frac{nm}{(n+m)^2}\sum_{i=1}^{N}(F_n^*(Z_i^*)-G_m^*(Z_i^*))^2.
		\end{align*}
		\item[(iii) ] Two sample AD test statistic based on data depth:
		\begin{align*}
			\tilde{dAD}=&\frac{nm}{n+m}\int_{-\infty}^{\infty}\left(F_n^*(t)-G_m^*(t)\right)^2\frac{1}{H_N^*(t)(1-H_N^*(t))}dH_N^*(t)\\
			=&\frac{nm}{(n+m)^2}\sum_{i=1}^{N}(F_n^*(Z_i^*)-G_m^*(Z_i^*))^2\frac{1}{H_N^*(Z_i^*)(1-H_N^*(Z_i^*))}.
		\end{align*}
	\end{itemize}
	Suppose that the ranks of $\{X_1^*,X_2^*,\ldots,X_n^*\}$ and $\{Y_1^*,Y_2^*,\ldots,Y_m^*\}$ in the joint sample $\{Z_1^*,Z_2^*,\ldots,Z_N^*\}$ be $\{R_{11},R_{12},\ldots,R_{1n}\}$ and $\{R_{21},R_{22},\ldots,R_{2n}\}$, respectively. Using the information from ranks, the above mentioned statistics can be re-written as
	\begin{align*}
		\tilde{dKS}=&\max_{1\leq j\leq N}\bigg|\frac{1}{n}\sum_{i=1}^{n}\mathbb{I}(R_{1i}\leq j) -\frac{1}{m}\sum_{i=1}^{m}\mathbb{I}(R_{2i}\leq j)\bigg|,\\
		\tilde{dCvM}=&\frac{nm}{(n+m)^2}\sum_{j=1}^{N} \left(\frac{1}{n}\sum_{i=1}^{n}\mathbb{I}(R_{1i}\leq j) -\frac{1}{m}\sum_{i=1}^{m}\mathbb{I}(R_{2i}\leq j) \right)^2,\\
		\tilde{dAD}=&\frac{nm}{(n+m)^2}\sum_{j=1}^{N-1}\frac{1}{j(N-j)} \left(\frac{1}{n}\sum_{i=1}^{n}\mathbb{I}(R_{1i}\leq j) -\frac{1}{m}\sum_{i=1}^{m}\mathbb{I}(R_{2i}\leq j) \right)^2
	\end{align*}
	
	Depth ranking may lead to ties and hence tie braking becomes important. We use random tie breaking scheme for numerical study. Other tie breaking scheme discussed by \cite{liu2006} can also be used. Under random tie breaking scheme we have following result.
	\begin{thm}
		Under $H_0$, statistics $\tilde{dKS}$, $\tilde{dCvM}$ and $\tilde{dAD}$ are distribution-free and have the same distribution as usual two sample $KS$, $CvM$ and $AD$ test statistics, respectively.
	\end{thm}
	Theorem 4 conveys that the proposed test statistics $\tilde{dKS}$, $\tilde{dCvM}$ and $\tilde{dAD}$ are distribution-free and have same distribution as usual two sample $KS$, $CvM$ and $AD$ test statistics, respectively.
	\section{Numerical Study}
	For assessing the finite sample performances of the proposed tests, we perform following simulation studies. Empirical powers are reported based on $1000$ replicates.
	\subsection{One sample GoF tests}
	In this section, we report results on the small sample performance of proposed tests. We compare their empirical powers with those of multivariate KS (MKS) test ({\cite{ks1997}}) and empirical characteristic function (ECF) based GoF proposed by {\cite{cfgof2009}} for bivariate cases for which algorithms are available. We use an algorithm proposed by {\cite{cfgof2009}} for ECF based GoF test for testing Morgenstern's bivariate distribution. For this purpose, we use true parameter value in place of estimated one. The ECF based GoF test is denoted by ECFT. Implementation of the MKS test for $d=2$ is carried out using the algorithm provided by {\cite{ks1997}}, and we denote this test by MKST. For dKS and dCvM tests based on $G^D_{N,n}$, estimated Type-I error rates are close to $0.05$ at level $0.05$ (see Table 1,3 and 6) for $N=5000$, other test statistics (AD, Greenwood, etc.) require larger $N$ for this. So we consider only KS and CvM test for empirical power study. We denote $dKS$ and $dCvM$ test based on Tukey's depth and zonoid depth by $tdKS$ and $tdCvM$, and $zdKS$ and $zdCvM$, respectively.
	
	Indeed, the power of any GoF test depends on the sample size. For numerical power study, we consider sample sizes $n=10,\ 25,\ 50,\ 100$ and $200$.\\~\\
	\textit{Example 1:}
	Here, we consider $N_2(\textbf{0},I_2)$ as the null distribution. Empirical powers are computed against alternatives $N_2(\mu,I_2)$, $N_2(\mathbf{0},\Sigma)$, $N_2(\mu, \Sigma)$, $t_1(\mathbf{0},I_2)$, $t_5(\textbf{0},I_2)$ and $N_2(\textbf{0},\Sigma_1)$, where $$\mathbf{0}= \begin{pmatrix}
	0 \\
	0 \\
	\end{pmatrix},\ \mu= \begin{pmatrix}
	1 \\
	1 \\
	\end{pmatrix},\ I_2= \begin{pmatrix}
	1 & 0 \\
	0 & 1 \\
	\end{pmatrix},\text{ and } \Sigma= \begin{pmatrix}
	1.5 & 0 \\
	0 & 1.5 \\
	\end{pmatrix}.$$
	Empirical powers of these are compared with those of the MKS test. Empirical power of MKS test is computed using the algorithm provided by {\cite{ks1997}}. ECF based test algorithm  is not available for this case.
	
	Estimated Type-I error rates for the proposed tests are given in {Table 1}, which remains close to the nominal value; estimated powers against the aforementioned alternatives are given in {Table 2}. For location difference, MKST performs better than the proposed tests for small sample size, but for large samples, the proposed tests have empirical powers close to that of MKST. For heavy-tailed alternatives (like multivariate Cauchy and multivariate t distributions), the proposed tests are superior to MKST.
	Empirical powers of GoF tests (KS and CvM) based on half space depth and zonoid depth are quite similar and it can not be concluded that GoF test based on which data depth is superior. But, the performance of CvM test based on data depth  is superior to the KS test based on data depth (half space or zonoid).
	\begin{table}
		\centering
		\caption{ {Estimated Type-I error probability for testing $N_2(\mathbf{0}, I_2)$ at level $0.05$.}}
		\begin{tabular}{c c c c c c }
			\hline
			Test & n=10 & 25 & 50 & 100 & 200 \\
			\hline
			$tdKS$ & 0.046 & 0.049 & 0.044 & 0.048 & 0.052 \\
			$tdCvM$ & 0.051 & 0.044 & 0.053 & 0.048 & 0.056 \\
			$zdKS$ & 0.050 & 0.045 & 0.052 & 0.054 & 0.053 \\
			$zdCvM$ & 0.053 & 0.050 & 0.056 & 0.055 & 0.047\\
			\hline
		\end{tabular}
	\end{table}
	\begin{table}
		\centering
		\caption{ {Empirical power of tests for $N_2(\mathbf{0},I_2)$ at level $0.05$.}}
		\begin{tabular}{c c c c c c c }
			\hline
			Alternative & Test &n=10 & 25 & 50 & 100 & 200 \\
			\hline
			$N_2(\mu,I_2)$ & $tdKS$ & 0.533 & 0.858 & 0.989 & 1.000 & 1.000 \\
			& $tdCvM$ & 0.598 & 0.911 & 0.997 & 1.000 & 1.000 \\
			& $zdKS$ & 0.535 & 0.862 & 0.990 & 1.000 & 1.000 \\
			& $zdCvM$ & 0.595 & 0.911 & 0.997 & 1.000 & 1.000 \\
			& $MKST$ & 0.962 & 1.000 & 1.000 & 1.000 & 1.000 \\
			\hline
			$N_2(\textbf{0},\Sigma)$& $tdKS$ & 0.180 & 0.353 & 0.584 & 0.890 & 0.992 \\
			& $tdCvM$ & 0.212 & 0.409 & 0.676 & 0.921 & 0.997 \\
			& $zdKS$ & 0.181 & 0.349 & 0.590 & 0.893 & 0.992 \\
			& $zdCvM$ & 0.211 & 0.404 & 0.677 & 0.922 & 0.997 \\
			& $MKST$ & 0.111 & 0.121 & 0.165 & 0.318 & 0.600 \\
			\hline
			$N_2(\mu,\Sigma)$ & $tdKS$ & 0.671 & 0.964 & 1.000 & 1.000 & 1.000 \\
			& $tdCvM$ & 0.722 & 0.988 & 1.000 & 1.000 & 1.000 \\
			& $zdKS$ & 0.676 & 0.963 & 1.000 & 1.000 & 1.000 \\
			& $zdCvM$ & 0.726 & 0.987 & 1.000 & 1.000 & 1.000 \\
			& $MKST$ & 0.937 & 1.000 & 1.000 & 1.000 & 1.000 \\
			\hline
			$t_1(\textbf{0},I_2)$& $tdKS$ & 0.420 & 0.870 & 1.000 & 1.000 & 1.000 \\
			& $tdCvM$ & 0.476 & 0.850 & 0.996 & 1.000 & 1.000 \\
			& $zdKS$ & 0.420 & 0.868 & 0.998 & 1.000 & 1.000 \\
			& $zdCvM$ & 0.478 & 0.850 & 0.996 & 1.000 & 1.000 \\
			& $MKST$ & 0.158 & 0.374 & 0.668 & 0.978 & 1.000 \\
			\hline
			$t_5(\textbf{0},I_2)$ & $tdKS$ & 0.112 & 0.168 & 0.220 & 0.436 & 0.756 \\
			& $tdCvM$ & 0.118 & 0.202 & 0.246 & 0.446 & 0.760 \\
			& $zdKS$ & 0.114 & 0.172 & 0.226 & 0.426 & 0.754 \\
			& $zdCvM$ & 0.118 & 0.200 & 0.244 & 0.442 & 0.762 \\
			& $MKST$ & 0.086 & 0.082 & 0.056 & 0.108 & 0.128 \\
			\hline
		\end{tabular}
	\end{table}
	\noindent
	\textit{Example 2:} Morgenstern's system of bivariate distributions with marginal distributions $F$ and $G$ is given by
	$$\mathcal{F}(x,y;\theta)=F(x)G(y)[1+\theta \{1-F(x)\}\{1-G(y)\}],\ |\theta|\leq 1.$$
	Here we consider Morgenstern's bivariate distribution with uniform marginals and $\theta=0,$ and $0.50$ as null distributions. Empirical powers of proposed tests are computed against Morgenstern's bivariate alternatives with Beta marginals and compared with those of the multivariate KS test.
	
	Empirical Type-I error rate are given in {Table 3} and empirical powers for various alternatives are given in {Tables 4-5}. $U[0,1]^2$ corresponds to Morgenstern's bivariate distribution with uniform marginals and $\theta=0$. In this case, we compare empirical powers of the proposed tests with MKS test and ECF based GoF test. Empirical power of MKST is computed using the algorithm provided by \cite{ks1997}. {\cite{cfgof2009}} discussed ECF based GoF test and provided algorithm for the test when underlying null distribution is Morgenstern's bivariate distribution, we use this algorithm to the compute empirical powers. For this numerical power study, we consider bivariate Morgenstern distribution with $U[0,\ 1]$ marginals as the null distribution, and alternatives as Morgenstern's bivariate distributions with marginals $Beta(a_1,b_1)$ and $Beta(a_2,b_2)$. We have computed Type-I error rates and numerical powers for $\theta\in\ \{0,0.25,0.50,0.75\}$ and $a_1=a_2=b_1=b_2=0.5$; $a_1=a_2=b_1=b_2=0.5$; and $a_1=a_2=2,$ and $b_1=b_2=3$.
	
	In this case, the Type-I error rate remains close to the nominal value $0.05$ as well. For all alternatives, the proposed tests perform better than MKS test and ECF based GoF test. The distribution $Beta(1.5,1.5)$ is close to $U[0,1]$, for the alternative with marginals $Beta(1.5,1.5)$, MKS test and ECF based GoF test have very low power in small sample case but the proposed tests show satisfactory power in this case also. {\cite{cfgof2009}} had observed that ECF based GoF tests may not be unbiased in the present test when the alternative has marginals $Beta(1.5,1.5)$ However, the proposed tests do not suffer from such anomaly.
	
	\begin{table}
		\centering
		\caption{ {Estimated Type-I error probability for testing Morgenstern's bivariate distribution with uniform marginals and different $\theta $ at level $0.05$.}}
		\begin{tabular}{c c c c c c c}
			\hline
			&Test & n=10 & 25 & 50 & 100 & 200 \\
			\hline
			$\theta$ =\ 0 &$tdKS$ & 0.037 & 0.041 & 0.046 & 0.056 & 0.04 \\
			&$tdCvM$ & 0.042 & 0.038 & 0.054 & 0.05 & 0.054 \\
			&$zdKS$ & 0.039 & 0.047 & 0.045 & 0.057 & 0.042 \\
			&$zdCvM$ & 0.041 & 0.043 & 0.054 & 0.051 & 0.055 \\
			&$ECFT$ & 0.053 & 0.048 & 0.051 & 0.046 & 0.041 \\
			\hline
			$\theta$ =\ 0.50 &$tdKS$ & 0.042 & 0.042 & 0.058 & 0.052 & 0.054 \\
			&$tdCvM$ & 0.040 & 0.038 & 0.066 & 0.058 & 0.060 \\
			&$zdKS$ & 0.044 & 0.044 & 0.062 & 0.056 & 0.056 \\
			&$zdCvM$ & 0.038 & 0.036 & 0.064 & 0.056 & 0.056 \\
			&$ECFT$ & 0.058 & 0.060 & 0.062 & 0.028 & 0.036 \\
			\hline
		\end{tabular}
	\end{table}
	\begin{table}
		\centering
		\caption{ {Empirical power of tests for $U[0,1]^2$ at level $0.05$ for Morgenstern's bivariate alternatives with $beta(a_1,b_1)$ \& $beta(a_2,b_2)$ marginals and $\theta = 0$.}}
		\begin{tabular}{c c c c c c c }
			\hline
			Alternative & Test &n=10 & 25 & 50 & 100 & 200 \\
			\hline
			$a_1=a_2=0.5$ & $tdKS$ & 0.574 & 0.874 & 0.994 & 1.000 & 1.000 \\
			$b_1=b_2=0.5$ & $tdCvM$ & 0.652 & 0.910 & 1.000 & 1.000 & 1.000 \\
			& $zdKS$ & 0.574 & 0.878 & 0.996 & 1.000 & 1.000 \\
			& $zdCvM$ & 0.650 & 0.910 & 1.000 & 1.000 & 1.000 \\
			& $ECFT$ & 0.172 & 0.352 & 0.744 & 0.984 & 1.000 \\
			& $MKST$ & 0.198 & 0.338 & 0.680 & 0.944 & 1.000 \\
			\hline
			$a_1=a_2=1.5$ & $tdKS$ & 0.190 & 0.440 & 0.768 & 0.972 & 1.000 \\
			$b_1=b_2=1.5$ & $tdCvM$ & 0.208 & 0.492 & 0.808 & 0.990 & 1.000 \\
			& $zdKS$ & 0.188 & 0.448 & 0.766 & 0.976 & 1.000 \\
			& $zdCvM$ & 0.212 & 0.492 & 0.810 & 0.990 & 1.000 \\
			& $ECFT$ & 0.036 & 0.052 & 0.124 & 0.406 & 0.890 \\
			& $MKST$ & 0.036 & 0.060 & 0.144 & 0.328 & 0.792 \\
			\hline
			$a_1=a_2=2$ & $tdKS$ & 0.460 & 0.918 & 0.996 & 1.000 & 1.000 \\
			$b_1=b_2=3$ & $tdCvM$ & 0.538 & 0.958 & 1.000 & 1.000 & 1.000 \\
			& $zdKS$ & 0.464 & 0.932 & 0.996 & 1.000 & 1.000 \\
			& $zdCvM$ & 0.536 & 0.958 & 1.000 & 1.000 & 1.000 \\
			& $ECFT$ & 0.342 & 0.932 & 1.000 & 1.000 & 1.000 \\
			& $MKST$ & 0.430 & 0.976 & 1.000 & 1.000 & 1.000 \\
			\hline
		\end{tabular}
	\end{table}
	\begin{table}
		\centering
		\caption{ {Empirical power of tests for Morgenstern's bivariate alternatives with $U[0,1]$ marginals, $\theta = 0.50$ at level $0.05$ for Morgenstern's bivariate alternatives with $beta(a_1,b_1)$ \& $beta(a_2,b_2)$ marginals and $\theta = 0.50$.}}
		\begin{tabular}{c c c c c c c }
			\hline
			Alternative & Test &n=10 & 25 & 50 & 100 & 200 \\
			\hline
			$a_1=a_2=0.5$ & $tdKS$ & 0.544 & 0.894 & 0.998 & 1.000 & 1.000 \\
			$b_1=b_2=0.5$ & $tdCvM$ & 0.610 & 0.940 & 1.000 & 1.000 & 1.000 \\
			& $zdKS$ & 0.544 & 0.896 & 0.998 & 1.000 & 1.000 \\
			& $zdCvM$ & 0.614 & 0.940 & 1.000 & 1.000 & 1.000 \\
			& $ECFT$ & 0.170 & 0.194 & 0.270 & 0.302 & 0.426 \\
			& $MKST$ & 0.236 & 0.372 & 0.624 & 0.954 & 1.000 \\
			\hline
			$a_1=a_2=1.5$ & $tdKS$ & 0.196 & 0.430 & 0.758 & 0.966 & 0.998 \\
			$b_1=b_2=1.5$ & $tdCvM$ & 0.222 & 0.522 & 0.840 & 0.986 & 0.998 \\
			& $zdKS$ & 0.202 & 0.432 & 0.760 & 0.968 & 0.998 \\
			& $zdCvM$ & 0.232 & 0.524 & 0.844 & 0.986 & 0.998 \\
			& $ECFT$ & 0.032 & 0.048 & 0.100 & 0.126 & 0.282 \\
			& $MKST$ & 0.044 & 0.068 & 0.088 & 0.268 & 0.714 \\
			\hline
			$a_1=a_2=2$ & $tdKS$ & 0.558 & 0.918 & 1.000 & 1.000 & 1.000 \\
			$b_1=b_2=3$ & $tdCvM$ & 0.634 & 0.954 & 1.000 & 1.000 & 1.000 \\
			& $zdKS$ & 0.550 & 0.914 & 1.000 & 1.000 & 1.000 \\
			& $zdCvM$ & 0.642 & 0.960 & 1.000 & 1.000 & 1.000 \\
			& $ECFT$ & 0.088 & 0.456 & 0.874 & 0.998 & 1.000 \\
			& $MKST$ & 0.352 & 0.936 & 1.000 & 1.000 & 1.000 \\
			\hline
		\end{tabular}
	\end{table}
	\noindent
	\textit{Example 3:} Here, we consider $N_5(\textbf{0}_5,I_5)$ as the null distribution. Empirical powers are computed against alternatives $N_5(\mu_1,I_5)$, $N_5(\textbf{0}_5,1.5I_5)$, $N_5(\mu_1,1.5I_5)$, $t_1(\textbf{0}_5,I_5)$ and $t_5(\textbf{0}_5,I_5)$, where $\textbf{0}_5$ is the 5-dimensional column vector with zero entries, $\mu_1$ is the 5-dimensional column vector with entries one and $I_5$ is the 5-dimensional identity matrix.
	
	Estimated Type-I error rates for proposed tests are given in {Table 6}, which remain close to the nominal value. Estimated powers against aforementioned alternatives are given in {Table 7}. No algorithm is available when dimension is greater than two for either MKS, or ECF based multivariate GoF tests. So, we have not compared empirical power of proposed tests with any other test. Numerical power shows that the proposed tests continue to have satisfactory power for this case as well.
	\begin{table}
		\centering
		\caption{Estimated Type-I error probability for testing $N_5(\textbf{0}_5, I_5)$, $N=5000$ at level $0.05$.}
		\begin{tabular}{c c c c c c }
			\hline
			Test & n=10 & 25 & 50 & 100 & 200 \\
			\hline
			$tdKS$ & 0.042 & 0.054 & 0.042 & 0.040 & 0.050 \\
			$tdCvM$ & 0.040 & 0.058 & 0.046 & 0.054 & 0.053 \\
			$zdKS$ & 0.052 & 0.052 & 0.043 & 0.048 & 0.056 \\
			$zdCvM$ & 0.058 & 0.056 & 0.045 & 0.052 & 0.054 \\
			\hline
		\end{tabular}
	\end{table}
	\begin{table}
		\centering
		\caption{Empirical power of tests for $N_2(\textbf{0}_5,I_5)$, $N=5000$ at level $0.05$.}
		\begin{tabular}{c c c c c c c }
			\hline
			Alternative & Test &n=10 & 25 & 50 & 100 & 200 \\
			\hline
			$N_5(\mu,I_5)$ & $tdKS$ & 0.848 & 0.996 & 1.000 & 1.000 & 1.000 \\
			& $tdCvM$ & 0.924 & 0.998 & 1.000 & 1.000 & 1.000 \\
			& $zdKS$ & 0.844 & 0.996 & 1.000 & 1.000 & 1.000 \\
			& $zdCvM$ & 0.922 & 0.998 & 1.000 & 1.000 & 1.000 \\
			\hline
			$N_5(0_5, 1.5I_5)$ & $tdKS$ & 0.314 & 0.704 & 0.964 & 1.000 & 1.000 \\
			& $tdCvM$ & 0.424 & 0.812 & 0.986 & 1.000 & 1.000 \\
			& $zdKS$ & 0.316 & 0.724 & 0.96 & 1.000 & 1.000 \\
			& $zdCvM$ & 0.432 & 0.818 & 0.984 & 1.000 & 1.000 \\
			\hline
			$N_5(\mu,1.5I_5)$ & $tdKS$ & 0.950 & 1.000 & 1.000 & 1.000 & 1.000 \\
			& $tdCvM$ & 0.972 & 1.000 & 1.000 & 1.000 & 1.000 \\
			& $zdKS$ & 0.950 & 1.000 & 1.000 & 1.000 & 1.000 \\
			& $zdCvM$ & 0.976 & 1.000 & 1.000 & 1.000 & 1.000 \\
			\hline
			$t_1(0_5, I_5)$ & $tdKS$ & 0.614 & 0.976 & 1.000 & 1.000 & 1.000 \\
			& $tdCvM$ & 0.65 & 0.976 & 1.000 & 1.000 & 1.000 \\
			& $zdKS$ & 0.606 & 0.982 & 1.000 & 1.000 & 1.000 \\
			& $zdCvM$ & 0.648 & 0.976 & 1.000 & 1.000 & 1.000 \\
			\hline
			$t_5(0_5, I_5)$ & $tdKS$ & 0.174 & 0.33 & 0.624 & 0.916 & 1.000 \\
			& $tdCvM$ & 0.202 & 0.356 & 0.6 & 0.884 & 1.000 \\
			& $zdKS$ & 0.172 & 0.368 & 0.712 & 0.968 & 1.000 \\
			& $zdCvM$ & 0.212 & 0.378 & 0.616 & 0.892 & 1.000 \\
			\hline
			
		\end{tabular}
	\end{table}
	\vspace*{-0.1in}
	\subsection{Two sample tests}
	We now compare the proposed two sample tests with tests proposed by \cite{liu2006} and \cite{li2016}. For this simulation study, we take $m=n=100$. Tests proposed by \cite{liu2006} and \cite{li2016} are denoted by $\tilde{dW}$ and $\tilde{dS}$, respectively. The test $\tilde{dS}$ is based on bootstrap, and $1000$ bootstrap samples  are considered for the testing purpose.\\~\\
	\textit{Example 4:} In this example, we consider $F=N\left(\mathbf{0},I_2\right)$ and different $G$. For different alternatives, Table 8 gives the empirical powers of the tests. Here, $L(\mathbf{\mu}, \Sigma)$ denotes bivariate Laplace distribution with location parameter $\mathbf{\mu}$ and scale parameter $\Sigma$, and $t_d(\mathbf{\mu},\Sigma)$ denotes the multivariate t distribution with location parameter vector $\mathbf{\mu}$, scale parameter matrix $\Sigma$ and degrees of freedom $d$, where $\mathbf{\mu}\in\mathbb{R}^2$ and $\Sigma>0$ is $2\times2$ matrix.
	\begin{table}
		\centering
		\caption{Empirical powers of the tests for two sample tests at level $0.05$ with $F=N\left(\mathbf{0},I_2\right)$ and different $G$.}
		\begin{tabular}{llllll}
			\hline
			& \multicolumn{5}{c}{Tests} \\
			\cmidrule(lr){2-6}
			$G$& $\tilde{dKS}$ & $\tilde{dCvM}$ & $\tilde{dAD}$ & $\tilde{dW}$ & $\tilde{dS}$ \\ \hline
			$N(\mathbf{0}, I_2)$ & 0.051 & 0.050 & 0.053 & 0.054 & 0.052 \\
			$N(\mathbf{0}, 1.5I_2)$ & 0.580 & 0.609 & 0.600 & 0.654 & 0.740 \\
			$N(\mathbf{0}, 2I_2)$ & 0.957 & 0.972 & 0.974 & 0.983 & 1.000 \\ 
			$L(\mathbf{0}, I_2)$ & 0.728 & 0.851 & 0.885 & 0.650 & 0.652 \\
			$t_2(\mathbf{0}, I_2)$ & 0.807 & 0.734 & 0.850 & 0.638 & 0.835 \\
			\hline
		\end{tabular}
	\end{table}
	The first row of Table 8 shows that for all the tests, Type-I error rates are close to the nominal value. For normal alternatives, the test $\tilde{dS}$ performs better than other tests. For larger dispersion (or, heavy tailed alternatives), some of the proposed tests perform either the best or close to  $\tilde{dS}$. Specially performance of the test $\tilde{dAD}$ appears good for all heavy tailed alternatives.\\~\\~\\
	\textit{Example 5:} In this example, we take $F=t_1\left(\mathbf{0},I_2\right)$ and different $G$. For different alternatives, Table 9 gives the empirical powers of the tests.
	\begin{table}
		\centering
		\caption{Empirical powers of the tests for two sample tests at level $0.05$ with $F=t_1\left(\mathbf{0},I_2\right)$ and different $G$.}
		\begin{tabular}{llllll}
			\hline
			& \multicolumn{5}{c}{Tests} \\
			\cmidrule(lr){2-6}
			$G$& $\tilde{dKS}$ & $\tilde{dCvM}$ & $\tilde{dAD}$ & $\tilde{dW}$ & $\tilde{dS}$ \\ \hline
			$t_1(\mathbf{0}, I_2)$ & 0.046 & 0.05 & 0.052 & 0.05 & 0.051 \\
			$t_1(\mathbf{0},3I_2)$ & 0.855 & 0.913 & 0.914 & 0.902 & 0.995 \\ 
			$t_2(\mathbf{0}, I_2)$ & 0.423 & 0.474 & 0.559 & 0.397 & 0.674\\
			$t_5(\mathbf{0}, I_2)$ & 0.941 & 0.927 & 0.965 & 0.829 & 0.991\\
			$L(\mathbf{0}, I_2)$ & 0.995 & 0.997 & 0.999 & 0.993 & 0.999 \\
			\hline
		\end{tabular}
	\end{table}
	The first row of Table 8 indicates that Type-I error rates for all the tests are close to the nominal value.	For all alternatives, $\tilde{dCVM}$ and $\tilde{dAD}$ tests performances are close to  the best test  $\tilde{dS}$. So, considering time required for the tests, the tests $\tilde{dCVM}$ and $\tilde{dAD}$ should be preferred.
	
	Our simulation study suggests that proposed tests can perform reasonably well when the underlying distribution of one (or, both) samples are heavy-tailed. The test $\tilde{dS}$ performs the best most of the time but, it takes much more time to execute as compared to the other tests. Moreover, if there is high kurtosis difference between distributions of the samples (e.g., $N(\mathbf{0}, I_2)$ vs. $L(\mathbf{0}, I_2)$),  the test $\tilde{dAD}$ is performs better than all other considered tests.
	
	\section{Real data example}
	We now illustrate the performances of two sample tests on a classical tooth growth data set. This data set is available in R, under the library ``ToothGrowth".
	The data consists of the length of odontoblasts (cells responsible for tooth growth) in $60$ guinea pigs. There were three vitamin C dose levels and two delivery methods (VC and OJ). Each of the pigs received one of three dose levels of vitamin C by one of two delivery methods. There are $30$ observations corresponding to each delivery method. If we consider each delivery method corresponds to a population, then the data is suitable for our two sample test. From the depth-depth plot of the data (see Figure \ref{fig1}), it is evident that these two samples have difference in scales. We have used the proposed two sample tests to investigate whether these tests are able to detect this difference. Tukey depth is used for the analysis. The observed $p$-value of the tests are reported in Table 12. Since one co-ordinate of this data is discrete (which results into ties of computed depths), we used permutation method to compute the $p$-values.
	\begin{figure}[h]
		\centering
		\includegraphics[width=8cm,height=7cm]{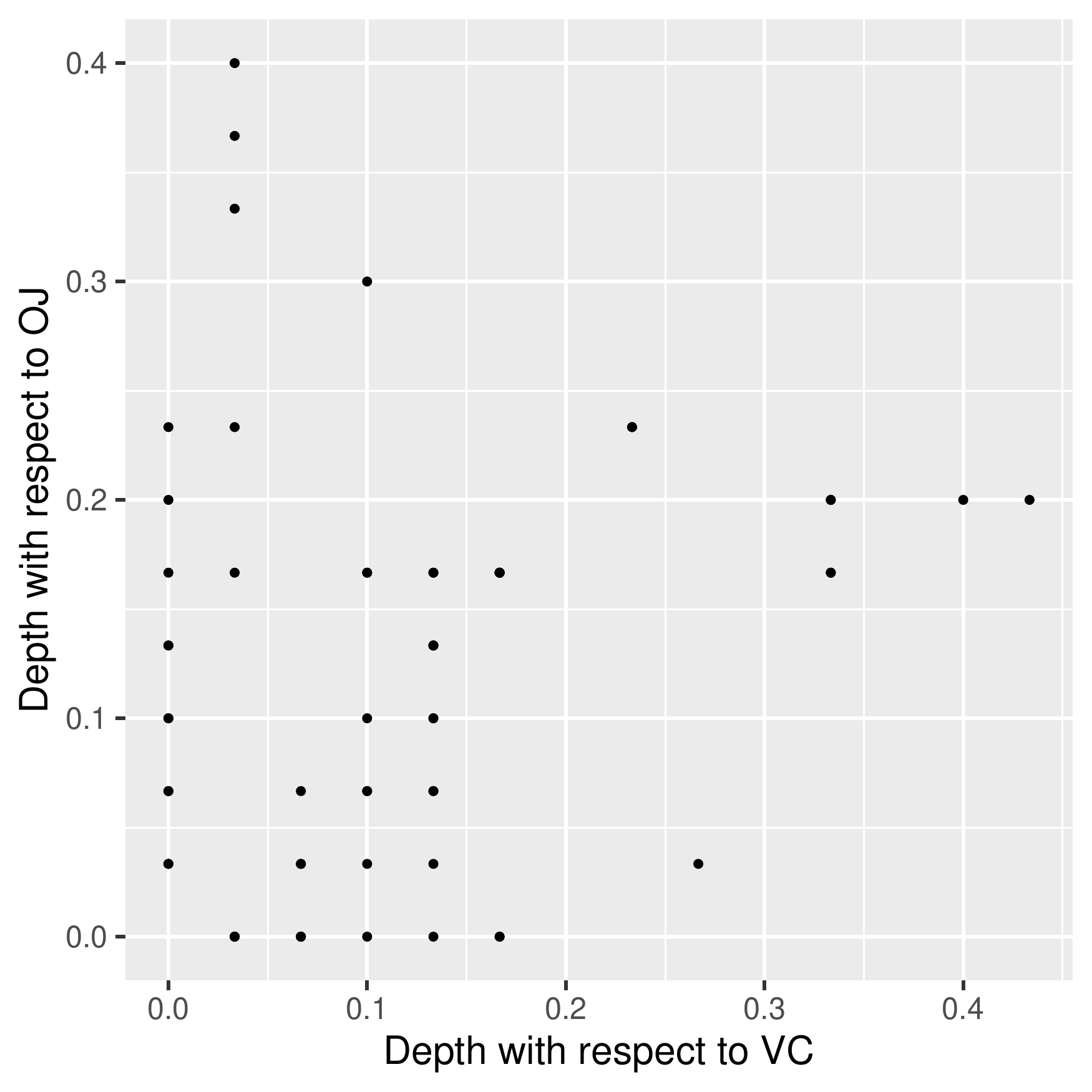}
		\caption{Depth-depth plot of the ``ToothGrowth" data.}\label{fig1}
	\end{figure}
	\begin{table}[h]
		\centering
		\caption{$p$-values of different tests for the ``ToothGrowth" data.}
		\begin{tabular}{llllll} \hline
			& \multicolumn{5}{c}{Tests} \\ \cmidrule(lr){2-6}
			& $\tilde{dKS}$ & $\tilde{dCvM}$ & $\tilde{dAD}$ & $\tilde{dW}$ & $\tilde{dS}$ \\ \cmidrule(lr){2-6}
			$p$-value & 0.253 & 0.115 & 0.03 & 0.40 & 0.394 \\
			\hline
		\end{tabular}
	\end{table}
	\FloatBarrier
	It is evident from the result in Table 10 that the proposed tests are quite useful. Different two sample test statistics measures different kind of departure of the populations. Hence, different tests are useful in different situations. One would need to apply a battery of tests in a given situation.
	\section{Concluding remarks}
	In this paper, we have studied several new multivariate GoF tests for continuous distribution functions. The proposed tests are based on notions of the popular Tukey's half-space and zonoid depths.
	For testing bivariate normality, $dKS$ and $dCvM$ tests perform better than the MKS test when there is a difference in scale (or, alternative is heavy-tailed). For location difference the MKS test performs better than $dKS$ and $dCvM$ tests for small sample size, while for large sample powers are comparable. Also, for testing Morgenstern's bivariate distribution with $U[0,1]$ marginals, $dKS$ and $dCvM$ tests perform better than MKS and ECF based GoF tests.
	For dimension greater than two, algorithm for implementation of MKS and ECF based GoF tests are not available in literature. Thus, implementation of these test procedures are difficult. Our proposed tests are applicable in higher dimensions and the algorithm is exactly the same. For testing pentavariate normality, the proposed tests have satisfactory powers.
	Empirical powers of the proposed GoF tests (KS and CvM) based on Tukey and zonoid depths are similar, and it can not be concluded that GoF test based on one data depth is superior to one based on the other data depth. But, the performance of CvM test based on data depth (half space or zonoid) is superior to KS test based on the data depth.
	
	We have also discussed some new two sample tests based on data depth for multivariate distributions. These tests are distribution-free under the null hypothesis. Simulation study suggests that the proposed tests are useful against some popular competitors for heavy-tailed distributions. Performance of the test $\tilde{dAD}$ is second best after the test $\tilde{dS}$. The test $\tilde{dS}$ is based on bootstrap, and such a test requires much more time to execute. The proposed tests do not suffer from such a disadvantage. When there is high kurtosis difference between distributions of the samples,  the test $\tilde{dAD}$ is performs better than all other considered tests.
	A real data example that conveys the same is presented. 
	\section*{Appendix}
	\begin{thm}[\cite{chung1949}]\label{lil}
		Let $F$ is continuous distribution function on $\mathbb{R}$ and $F_N$ be empirical distribution function corresponding to a random sample of size $N$ from $F$, then
		$$\limsup_{N\to \infty}\ \sup_{x\in \mathbb{R}}\ \sqrt{\frac{N}{\log \log N}}\ \big|F_N(x)-F(x)\big| \leq 1.$$
	\end{thm}
	
	\begin{lemma}
		Under assumptions (A1), if $n= o\left({\dfrac{N}{\log \log N}}\ \right)$ then \label{lemma1}
		\begin{small}
			$$||\mathbf{F}^D_{N,n}-\pmb(F^{D}_N(D(X_1)),F^{D}_N(D(X_2)),\ldots,F^{D}_N(D(X_n))\pmb)||\xrightarrow[N\to \infty]{a.s.}\ 0.$$
		\end{small}
	\end{lemma}
	\begin{proof}
		Observe that,
		\begin{align*}
			&||\mathbf{F}^D_{N,n}-\pmb(F^{D}_N(D(X_1)),F^{D}_N(D(X_2)),\ldots,F^{D}_N(D(X_n))\pmb)||\\
			\leq&~ \sqrt{n} \times \underset{D(y):x\in\mathbb{R}^d}{\sup}\ |F^D(D(x)) - F^{D}_N(D(x))|\\
			&\xrightarrow[N\to \infty]{\text{a.s.}}\ 0,\ \text{if } n= o\left({\dfrac{N}{\log\ \log N}}\ \right), \text{using Theorem \ref{lil} (stated above)} .
		\end{align*}
	\end{proof}
	\begin{lemma} Under assumptions (A1) and (A2), for any $y\in \mathbb{R}^d$ \label{lemma2}
		$$\big| G_N(x)-F_N^D(D(x))\big| \xrightarrow[N\to \infty]{a.s.} \ 0.$$
	\end{lemma}
	\begin{proof}
		Under the assumptions,
		$$\mathbb{I}[D_N(W_i)<\ D_N(x)]-\mathbb{I}[D(W_i)<\ D(x)]\xrightarrow[N\to \infty]{a.s.}\ 0.$$
		Also due to continuity,
		$$Pr[D(W_i)=D(x)]=0.$$
		
		Thus, we get,
		$$G_N(x)-F_N^D(D(x)) =\dfrac{1}{N}\sum_{i=1}^N\bigg\{\mathbb{I}[D_N(W_i)\leq\ D_N(x)]-\mathbb{I}[D(W_i)\leq\ D(x)]\bigg\}\xrightarrow[N\to \infty]{a.s.}\ 0.$$
		
	\end{proof}
	\begin{proof}[{Proof of Theorem 2}] Observe that
		\begin{align*}
			&||\mathbf{F}^D_{N,n}-\mathbf{G}^D_{N,n}||\\
			\leq& ||\mathbf{F}^D_{N,n}-\pmb(F^{D}_N(D(X_1)),F^{D}_N(D(X_2)),\ldots,F^{D}_N(D(X_n))\pmb)||\\
			&+ ||\pmb(F^{D}_N(D(X_1)),F^{D}_N(D(X_2)),\ldots,F^{D}_N(D(X_n))\pmb)-\mathbf{G}^D_{N,n}||\\
			&\xrightarrow[N\to \infty]{a.s.}\ 0,\ \text{due to Lemmas \ref{lemma1} and \ref{lemma2}.}
		\end{align*}
	\end{proof}
	\begin{proof}[{Proof of Theorem 3}]
		Let $u\in\ [0,1]$. Define $G_{an}(u):=\ \dfrac{1}{n}\sum_{i=1}^n\mathbb{I}[F^D(D(X_i))\leq u]$, \\
		and $H_n(u):=\ \dfrac{1}{n}\sum_{i=1}^n\mathbb{I}[G_N(X_i)\leq u]$.
		
		Observe that
		\begin{align*}
			\sqrt{n}\ \big|G_{an}(u)-u\big| & \leq \sqrt{n}\ \big|G_{an}(u)-H_n(u)\big| + \sqrt{n}\ \big|H_n(u)-u\big| \\
			& = \frac{1}{\sqrt{n}}\bigg|\sum_{i=1}^n\bigg\{\mathbb{I}[F^{D}(D(X_i))\leq u]-\mathbb{I}[G_N(X_i)\leq u]\bigg\}\bigg|
			+ \sqrt{n}\ \big|H_n(u)-u\big|
		\end{align*}
		\begin{align*}
			\Rightarrow\sqrt{n}\ \big|G_{an}(u)-u\big| - \sqrt{n}\ \big|H_n(u)-u\big| \leq\
			\frac{1}{\sqrt{n}}\bigg|\sum_{i=1}^n\bigg\{\mathbb{I}[F^{D}(D(X_i))\leq u]-\mathbb{I}[G_N(X_i)\leq u]\bigg\}\bigg|.
		\end{align*}
		Similarly, for any $u\in(0,1)$, we have
		\begin{align*}
			&\sqrt{n}\ |H_n(u)-u|- \sqrt{n}\ |G_{an}(u)-u| \leq\ \frac{1}{\sqrt{n}}\bigg|\sum_{i=1}^n\bigg\{\mathbb{I}[F^{D}(D(X_i))\leq u]-\mathbb{I}[G_N(X_i)\leq u]\bigg\}\bigg|\\
			&\Rightarrow \pmb{\big|}\sqrt{n}\ |H_n(u)-u|- \sqrt{n}\ |G_{an}(u)-u|\pmb{\big|} \leq\ \frac{1}{\sqrt{n}}\bigg|\sum_{i=1}^n\bigg\{\mathbb{I}[F^{D}(D(X_i))\leq u]-\mathbb{I}[G_N(X_i)\leq u]\bigg\}\bigg|.
		\end{align*}
		Using Theorem \ref{thm2} and arguments similar to proof of Lemma \ref{lemma2}, it can be shown that
		\begin{align*}
			&\frac{1}{\sqrt{n}}\bigg|\sum_{i=1}^n\bigg\{\mathbb{I}[F^{D}(D(X_i))\leq u]-\mathbb{I}[G_N(X_i)\leq u]\bigg\}\bigg|
			\xrightarrow[N\to \infty]{a.s.}0\\
			\Rightarrow& \pmb{\big|}\sqrt{n}\ |H_n(u)-u|- \sqrt{n}\ |G_{an}(u)-u|\pmb{\big|} \xrightarrow[N\to \infty]{a.s.}0.
		\end{align*}
		Using the continuous mapping theorem and Slutsky's lemma, the proof is complete for the KS test statistic. Proof for AD, CvM and Greenwood test statistics follows in a similar way.
	\end{proof}
	\begin{proof}[Proof Theorem 4]
		Observe that under $H_0$, the joint distribution of $(R_{11},R_{12},\ldots,R_{1n},R_{21},R_{22},\ldots,R_{2n})^T$ is uniform over all permutations of $\{1,2,\ldots,N\}$. Thus the proposed test statistics $\tilde{dKS}$, $\tilde{dCvM}$ and $\tilde{dAD}$ are distribution-free and have same distribution as the usual two sample $KS$, $CvM$ and $AD$ test statistics, respectively.
	\end{proof}
	\bibliographystyle{apa}
	\small
	\bibliography{bibpaper}
	
\end{document}